\documentclass[11pt]{amsart}
\usepackage{amsmath,amssymb}
\usepackage{savesym}
\usepackage[alphabetic, msc-links, bibtex-style, nobysame]{amsrefs}
\newtheorem{theorem}{{Theorem}}
\newtheorem*{theorem*}{Theorem}

\newtheorem{proposition}[theorem]{Proposition}
\newtheorem{corollary}[theorem]{Corollary}
\newtheorem*{corollary*}{Corollary}

\theoremstyle{definition}

\theoremstyle{remark}

\usepackage{lmodern,url,enumerate,mathtools,microtype}
\usepackage{graphicx}

\usepackage{xcolor}
\definecolor{ForestGreen}{RGB}{34,139,34}

\usepackage[hmargin = 1.15in,vmargin=0.95in]{geometry}
\usepackage{graphicx}
\usepackage{subcaption}

\usepackage{hyperref}
    \hypersetup{colorlinks=true,citecolor=ForestGreen,linkcolor = blue,urlcolor =black,linkbordercolor={1 0 0}}

\theoremstyle{definition}
\newtheorem*{remark*}{Remark}
\newtheorem*{idea*}{Idea}
\newtheorem*{definition*}{Definition}
\newtheorem*{claim*}{Claim}

\newcommand{\nc}{\newcommand}
\nc{\on}{\operatorname}

\newcommand{\p}{\partial}

\title[Fill-ins with scalar curvature bounded from below]{On fill-ins with scalar curvature bounded from below and an inequality of Hijazi--Montiel--Rold\'an}
\date{\today}

\author[Simon Brendle]{Simon Brendle}
\address{
{\href{mailto:simon.brendle@columbia.edu}{simon.brendle@columbia.edu}}}
\author[Raphael Tsiamis]{Raphael Tsiamis}
\address{
{\href{mailto:r.tsiamis@columbia.edu}{r.tsiamis@columbia.edu}}}
\author[Yipeng Wang]{Yipeng Wang}
\address{
{\href{mailto:yw3631@columbia.edu}{yw3631@columbia.edu}}}

\begin{document}

\begin{abstract}
We consider fill-ins of spin manifolds with scalar curvature bounded by $-n(n-1)$. Gromov proposed a conjecture relating the infimum of the mean curvature of such a fill-in to the hyperspherical radius. We observe that the inequality conjectured by Gromov follows by combining an inequality of Hijazi--Montiel--Rold\'an for the first Dirac eigenvalue with a recent theorem of B\"ar. Moreover, we give an alternative proof of the Hijazi--Montiel--Rold\'an inequality based on the work of B\"ar and B\"ar-Ballmann.
\end{abstract}

\maketitle

\vspace*{-0.22in}

\section{Introduction}

Fill-ins with scalar curvature bounded from below play a major role in geometry, allowing us to localize results for non-compact manifolds, such as the positive mass theorem.
In what follows, we let $\Sigma$ be a closed connected orientable Riemannian manifold. Hijazi--Montiel--Rold\'an~\cite{HMR}*{Theorem 2} proved the following inequality for the smallest Dirac eigenvalue.

\begin{theorem}[Hijazi--Montiel--Rold\'an]
\label{eigenvalue}
Let $(M^n, g)$ be a compact, connected Riemannian spin manifold of dimension $n\ge 3$ with connected boundary $\Sigma$.
If $R\geq -n(n-1)$ holds at each point of $M$, then
\[\inf_{\partial M}H\le \sqrt{(n-1)^2+4\lambda^2},\] 
where $\lambda = \sqrt{\lambda_{\text{\rm min}}(\mathcal{D}^2)}$ and $\mathcal{D}$ denotes the Dirac operator on the boundary. 
\end{theorem}

The proof given by Hijazi--Montiel--Rold\'an uses spinors. Specifically, they use boundary value problem for a modified Dirac operator. In this paper, we give an alternative proof of the result of Hijazi--Montiel--Rold\'an, which is based on a different boundary value problem for the Dirac operator. This argument relies on work of B\"ar \cite{Baer} and B\"ar-Ballmann \cite{bar-ballmann}.

Gromov and Lawson~\cite{gromov-lawson} introduced the following notion in their proof of Geroch's conjecture on the non-existence of positive scalar curvature metrics on tori. 
\begin{definition*}
The hyperspherical radius of $\Sigma$, denoted by $\on{Rad}(\Sigma)$, is defined as the supremum of all $r>0$ such that there exists a $1$-Lipschitz map $\Sigma \to \mathbf{S}^{\dim \Sigma}(r)$ with non-trivial degree.
\end{definition*}
The hyperspherical radius is a natural geometric quantity. For a spin manifold, Llarull's theorem \cite{Llarull1998} gives an upper bound for the infimum of the scalar curvature in terms of its hyperspherical radius. 

B\"ar~\cite{Baer} recently showed that, if $\Sigma$ is a compact spin manifold of dimension $n-1$, then 
\[\lambda \leq \frac{n-1}{2} \, \on{Rad} (\Sigma)^{-1},\] 
where again $\lambda = \sqrt{\lambda_{\text{\rm min}}(\mathcal{D}^2)}$. Combining the Hijazi--Montiel--Rold\'an inequality with B\"ar's inequality, one can draw the following conclusion.

\begin{corollary}
\label{main}
Let $(M^n, g)$ be a compact, connected Riemannian spin manifold of dimension $n\ge 3$ with connected boundary $\Sigma$.
If $R\ge -n(n-1)$ holds at each point of $M$, then
\[
\inf_{\partial M}H\le (n-1)\sqrt{1+{\on{Rad}(\Sigma)^{-2}}}.
    \]
\end{corollary}

Corollary \ref{main} answers a question posed by Gromov~\cite{gromov}*{p.~110}. It shares some common features with the systolic inequality proved by the first author and Pei-Ken Hung~\cite{brendle-hung-I}.

\section{An alternative proof of the Hijazi--Montiel--Rold\'an inequality}

In this section, we adapt an argument in B\"ar's paper \cite{Baer}*{Appendix A} to
give an alternative proof of the Hijazi--Montiel--Rold\'an inequality. Let $(M^n,g)$ be a compact spin manifold with $n\ge 3$, and $\mathcal{S}_M \to M$ be the spinor bundle. We denote by $\mathcal{D}^M$ the Dirac operator acting on sections of $\mathcal{S}_M$; in a local orthonormal frame $\{ e_1,\cdots,e_n \}$, the Dirac operator is given by $\mathcal{D}^M s=\sum_{k=1}^n e_k \cdot \nabla^M_{e_k} s$. 
Given a section $s\in C^{\infty}(M,\mathcal{S}_M)$ and a smooth vector field $X$ on $M$, we define
\[
\hat{\nabla}^M_X s=\nabla^M_Xs+\frac{i}{2} \, X\cdot s.
\] 
We denote by $\mathcal{D}^{\partial M}$ the boundary Dirac operator acting on sections of $\mathcal{S}_M|_{\partial M}$. Given a section $s$ of $\mathcal{S}_M|_{\partial M}$, this operator is defined by 
$$\mathcal{D}^{\partial M}s=\sum_{k=1}^{n-1} \nu\cdot e_k\cdot \nabla^M_{e_k} s+\frac{1}{2}H\,s,$$
where $H$ denotes the mean curvature of the boundary and $\{e_1,\cdots,e_{n-1}\}$ is a local orthonormal frame on $\p M$. 
It is standard that Clifford multiplication by $\nu$ anticommutes with $\mathcal{D}^{\partial M}$.

\begin{proposition}{\label{prop:weitzenbock}}
    For any $s\in C^{\infty}(M,\mathcal{S}_M)$, we have
    \begin{align*}
       &-\int_M \big | \mathcal{D}^M s-\frac{in}{2}s \big |^2+\int_M|\hat{\nabla}^M s|^2+\frac{1}{4}\int_M R\,|s|^2+\frac{n(n-1)}{4}\int_M |s|^2\\
    &\quad =\int_{\partial M}\langle \mathcal{D}^{\partial M}s,s\rangle -\frac{1}{2}\int_{\partial M}H\,|s|^2-\frac{n-1}{2}\int_{\partial M} \langle i\,\nu\cdot s,s\rangle
    \end{align*}
\end{proposition}

Proposition \ref{prop:weitzenbock} is well known. We omit the proof.

Let $\Sigma$ denote the boundary of $M$ with the induced metric. 
We denote by $\mathcal{S}$ and $\mathcal{D}$ the intrinsic spinor bundle and Dirac operator of $\Sigma$, respectively. 
If $n$ is odd, then $\mathcal{S}$ can be identified with $\mathcal{S}_M|_{\partial M}$, and $\mathcal{D}$ can be identified with $\mathcal{D}^{\partial M}$.
If $n$ is even, we decompose $\mathcal{S}_M = \mathcal{S}_M^+ \oplus \mathcal{S}_M^-$, where $\mathcal{S}_M^+$ and $\mathcal{S}_M^-$ denote the eigenbundles of the volume form on $M$. 
Clifford multiplication by $\nu$ interchanges $\mathcal{S}^+_M|_{\partial M}$ and $\mathcal{S}^-_M|_{\partial M}$. 
Moreover, $\mathcal{S}^+_M|_{\partial M}$ and $\mathcal{S}^-_M|_{\partial M}$ are invariant under the boundary Dirac operator $\mathcal{D}^{\partial M}$. 
With this understood, $\mathcal{S}$ can be identified with $\mathcal{S}^+_M|_{\partial M}$, and $\mathcal{D}$ can be identified with the restriction of $\mathcal{D}^{\partial M}$ to $\mathcal{S}^+_M|_{\partial M}$.
Thus, the boundary Dirac operator can be written in the form
\[\mathcal{D}^{\partial M}= \begin{bmatrix} \mathcal{D} & 0 \\ 0 & -\mathcal{D}
\end{bmatrix}.\]
For either parity, the spectrum of $\mathcal{D}^{\partial M}$ is symmetric about $0$, and we define 
\[\lambda = \sqrt{\lambda_{\text{\rm min}}((\mathcal{D}^{\partial M})^2)} =
\sqrt{\lambda_{\text{\rm min}}( \mathcal{D}^2)}\] 
under these identifications.

Let $\lambda_j$ denote the eigenvalues of $\mathcal{D}^{\partial M}$. By definition of $\lambda$, we have $|\lambda_j| \geq \lambda$ for each $j$. For each $j$, we denote by $E_{\lambda_j}$ the eigenspace of $\mathcal{D}^{\partial M}$ with eigenvalue $\lambda_j$. Any $s\in H^{\frac{1}{2}}(\p M,\mathcal{S}_M|_{\partial M})$ can be expanded as $s=\sum_j s_{\lambda_j}$, where $s_{\lambda_j} \in E_{\lambda_j}$. We define $W$ to be the subspace of all $s\in H^{\frac{1}{2}}(\p M,\mathcal{S}_M|_{\partial M})$ such that $s_{\lambda_j}=0$ for all $\lambda_j \in (\lambda,\infty)$. Moreover, we define $\tilde{W}$ to be the subspace of all $s\in H^{\frac{1}{2}}(\p M,\mathcal{S}_M|_{\partial M})$ such that $s_{\lambda_j}=0$ for all $\lambda_j \in [-\lambda,\infty)$. 

We consider the Hilbert spaces $\mathcal{H}=\{s\in H^1(M,\mathcal{S}_M): s|_{\partial M}\in W\}$ and $\tilde{\mathcal{H}}=\{s\in H^1(M,\mathcal{S}_M): s|_{\partial M}\in \tilde{W}\}$. The following result was proved by B\"ar \cite{Baer}*{Appendix A}, building on earlier work of B\"ar and Ballmann \cite{bar-ballmann}.

\begin{proposition}[C.~B\"ar \cite{Baer}, Appendix A]
\label{index}
The operator $\mathcal{D}^M:\mathcal{H}\to L^2(M,\mathcal{S}_M)$ is a Fredholm operator and its Fredholm index is given by 
\[\text{\rm ind}(\mathcal{D}^M:\mathcal{H}\to L^2(M,\mathcal{S}_M)) = \frac{1}{2} \sum_{\lambda_j \in [-\lambda,\lambda]} \dim_{\mathbb{C}} E_{\lambda_j}.\] 
\end{proposition}

\begin{proof}
We include the details of B\"ar's argument for the convenience of the reader. Corollary 8.8 in \cite{bar-ballmann} implies that 
\begin{align*} 
&\text{\rm ind}(\mathcal{D}^M:\mathcal{H}\to L^2(M,\mathcal{S}_M)) - \text{\rm ind}(\mathcal{D}^M:\tilde{\mathcal{H}}\to L^2(M,\mathcal{S}_M)) \\ 
&= \text{\rm dim}_{\mathbb{C}}(W/\tilde{W}) = \sum_{\lambda_j \in [-\lambda,\lambda]} \dim_{\mathbb{C}} E_{\lambda_j}. 
\end{align*}
On the other hand, 
\[\text{\rm ind}(\mathcal{D}^M:\tilde{\mathcal{H}}\to L^2(M,\mathcal{S}_M)) = -\text{\rm ind}(\mathcal{D}^M:\mathcal{H}\to L^2(M,\mathcal{S}_M))\] 
by duality (see \cite{bar-ballmann}*{Section 7.2}). Putting these facts together, the assertion follows. 
\end{proof}

Let $\mathcal{Z}$ denote the set of all sections $s \in H^1(M,\mathcal{S}_M)$ such that $\mathcal{D}^M s-\frac{in}{2}s=0$ and $s|_{\partial M} \in W$. Since the Fredholm index is unchanged by compact perturbations, we obtain 
\begin{align*} 
\dim_{\mathbb{C}} \mathcal{Z} 
&\geq \text{\rm ind} \Big( \mathcal{D}^M - \frac{in}{2}:\mathcal{H}\to L^2(M,\mathcal{S}_M) \Big ) \\ 
&= \text{\rm ind}(\mathcal{D}^M:\mathcal{H}\to L^2(M,\mathcal{S}_M)) \\ 
&= \frac{1}{2} \sum_{\lambda_j \in [-\lambda,\lambda]} \dim_{\mathbb{C}} E_{\lambda_j}. 
\end{align*}
In particular, $\dim_{\mathbb{C}} \mathcal{Z} > 0$. The results in \cite{bar-ballmann}*{Section 7.4} imply that every section $s \in \mathcal{Z}$ is smooth up to the boundary. 

\begin{proposition}
\label{rigidity}
Suppose that $s \in \mathcal{Z}$. Then 
\[\int_{\partial M} H \, |s|^2 \leq \sqrt{(n-1)^2+4\lambda^2}\,\int_{\partial M} |s|^2.\] 
Moreover, if equality holds, then $\hat{\nabla}^M s=0$ at each point in $M$. 
\end{proposition}

\begin{proof} 
We distinguish two cases:

\textit{Case 1:} Suppose that $\lambda=0$. Since $s|_{\partial M} \in W$, we know that 
\[\int_{\partial M} \langle \mathcal{D}^{\partial M} s, s \rangle \leq 0.\] 
Using Proposition~\ref{prop:weitzenbock}, we obtain 
\begin{align*}
    0&\leq \int_{\partial M}\langle \mathcal{D}^{\partial M}s,s\rangle -\frac{1}{2}\int_{\partial M}H\,|s|^2-\frac{n-1}{2}\int_{\partial M} \langle i\,\nu\cdot s,s\rangle\\
    &\leq -\frac{1}{2}\int_{\partial M}H\,|s|^2-\frac{n-1}{2}\int_{\partial M} \langle i\,\nu\cdot s,s\rangle.
\end{align*}
Applying the Cauchy-Schwarz inequality, we conclude that  
\[\int_{\partial M} H \, |s|^2 \leq (n-1) \int_{\partial M} |s|^2.\] 
Moreover, if equality holds, then $\int_M |\hat{\nabla}^M s|^2=0$.

\textit{Case 2:} Suppose that $\lambda>0$. Since $s|_{\partial M} \in W$, we may write $s|_{\partial M}=s_++s_-$, where $s_+ \in E_\lambda$ and $s_- \in \tilde{W} \oplus E_{-\lambda}$. This implies 
\[
\int_{\partial M} \langle \mathcal{D}^{\partial M} s, s \rangle = \int_{\partial M} \lambda \, | s_+ |^2 + \sum_{\lambda_j \in (-\infty,-\lambda]} \int_{\partial M} \lambda_j \, | s_{\lambda_j} |^2 \leq \int_{\partial M} \lambda |s_+|^2 - \int_{\partial M} \lambda |s_-|^2.
\]
Since Clifford multiplication by $\nu$ interchanges the positive and negative eigenspaces of $\mathcal{D}^{\partial M}$, we know that 
\[\int_{\partial M} \langle \nu \cdot s_+,s_+ \rangle = \int_{\partial M} \langle \nu \cdot s_-,s_- \rangle = 0.\]
Using Proposition~\ref{prop:weitzenbock}, we obtain 
\begin{align*}
    0&\leq \int_{\partial M}\langle \mathcal{D}^{\partial M}s,s\rangle -\frac{1}{2}\int_{\partial M}H\,|s|^2-\frac{n-1}{2}\int_{\partial M} \langle i\,\nu\cdot s,s\rangle\\
    &\leq \int_{\partial M} \lambda \, |s_+|^2-\int_{\partial M}\lambda \, |s_-|^2-\frac{1}{2}\int_{\partial M}H\,|s|^2-\frac{n-1}{2}\int_{\partial M} \langle i\,\nu\cdot s,s\rangle\\
    &=\int_{\partial M} \lambda \, |s_+|^2-\int_{\partial M}\lambda \, |s_-|^2-\frac{1}{2}\int_{\partial M}H\,|s|^2\\
    & \quad  -\frac{n-1}{2}\int_{\partial M} \langle i\,\nu\cdot s_+,s_-\rangle-\frac{n-1}{2}\int_{\partial M} \langle i\,\nu\cdot s_-,s_+\rangle.
\end{align*}
Applying the Cauchy-Schwarz inequality, we conclude that  
\begin{align*}
\int_{\partial M} H\,|s|^2&\le \int_{\partial M} (2\lambda\, (|s_+|^2-|s_-|^2)+2(n-1) \,  |s_+| \, |s_-|) \\
&\leq \int_{\partial M}\sqrt{(n-1)^2+4\lambda^2} \, \sqrt{(|s_+|^2-|s_-|^2)^2+4\,|s_+|^2\,|s_-|^2} \\
&=\sqrt{(n-1)^2+4\lambda^2}\,\int_{\partial M} (|s_+|^2 + |s_-|^2) \\ 
&=\sqrt{(n-1)^2+4\lambda^2}\,\int_{\partial M} |s|^2. 
\end{align*}
Moreover, if equality holds, then $\int_M |\hat{\nabla}^M s|^2=0$.
\end{proof}

\begin{corollary} 
\label{eigenvalue.estimate}
We have 
\[\inf_{\partial M}H\le \sqrt{(n-1)^2+4\lambda^2}.\] 
\end{corollary} 

\begin{proof} 
We argue by contradiction. Suppose that 
\[\inf_{\partial M}H > \sqrt{(n-1)^2+4\lambda^2}.\] 
Since $\dim_{\mathbb{C}} \mathcal{Z} > 0$, we can find a section $s \in H^1(M,\mathcal{S}_M)$ such that $\mathcal{D}^M s-\frac{in}{2}s=0$, $s|_{\partial M} \in W$, and $\int_M |s|^2 > 0$. 
Using Proposition \ref{rigidity}, we conclude that $\hat{\nabla}^M s=0$ in $M$ and $s|_{\partial M}=0$. This is a contradiction.  
\end{proof}

\begin{corollary} 
\label{hyperspherical.radius}
We have
\[\inf_{\partial M}H\le (n-1)\sqrt{1+{\on{Rad}(\Sigma)^{-2}}}.\] 
Moreover, if equality holds, then the spinor bundle $\mathcal{S}_M$ admits a trivialization by sections $s_1,\hdots,s_m \in C^{\infty}(M,\mathcal{S}_M)$ such that $\hat{\nabla}^M s_a=0$ for each $1\le a\le m$. Here, $m = 2^{[\frac{n}{2}]}$ denotes the rank of the spinor bundle $\mathcal{S}_M$.
\end{corollary}

\begin{proof} 
Combining Corollary \ref{eigenvalue.estimate} with B\"ar's inequality~\cite{Baer} 
\[\lambda = \sqrt{\lambda_{\text{\rm min}}( \mathcal{D}^2)} \leq \frac{n-1}{2} \, \on{Rad} (\Sigma)^{-1},
\] 
we obtain 
\[\inf_{\partial M}H\le (n-1)\sqrt{1+{\on{Rad}(\Sigma)^{-2}}}.\] 
Finally, we consider the case of equality. Suppose that 
\[\inf_{\partial M}H = (n-1)\sqrt{1+{\on{Rad}(\Sigma)^{-2}}}.\] 
Then 
\[\lambda = \sqrt{\lambda_{\text{\rm min}}( \mathcal{D}^2)} = \frac{n-1}{2} \, \on{Rad} (\Sigma)^{-1}.
\] 
By B\"ar's rigidity statement~\cite{Baer}, there exists some $r = \on{Rad}(\Sigma) >0$ such that $\Sigma$ is isometric to $\mathbf{S}^{n-1}(r)$ with the round metric. 
Since $\Sigma$ is isometric to a round sphere of dimension $n-1$, we know that $\lambda = \frac{n-1}{2r}$ and the $\pm \lambda$-eigenspaces of $\mathcal{D}$ each have dimension $2^{[\frac{n-1}{2}]}$ (cf. \cite{Baer1996_DiracPositiveCurvature}). 
From this, it is easy to see that the $\pm \lambda$-eigenspaces of $\mathcal{D}^{\partial M}$ each have dimension $2^{[\frac{n}{2}]}$. Thus, 
\[\dim_{\mathbb{C}} \mathcal{Z} \geq \frac{1}{2} \sum_{\lambda_j \in [-\lambda,\lambda]} \dim_{\mathbb{C}} E_{\lambda_j} \geq 2^{[\frac{n}{2}]}.\] 
Therefore, we can find a collection of linearly independent sections $s_1,\hdots,s_m \in \mathcal{Z}$, where $m = 2^{[\frac{n}{2}]}$. Since 
\[\inf_{\partial M}H = \sqrt{(n-1)^2+4\lambda^2},\] 
Proposition \ref{rigidity} implies that $\hat{\nabla}^M s_a=0$ for each $1\le a\le m$. Standard uniqueness results for ODE imply that $s_1|_p,\hdots,s_m|_p$ are linearly indepdendent for each point $p \in M$.
\end{proof}

\bibliography{ref}

\end{document}